\documentclass[11pt,reqno]{amsart}

\usepackage{amsmath, amsfonts, amsthm, amssymb, graphicx}
\usepackage{float,epsf,subfigure}

\theoremstyle{plain}
\newtheorem{theorem}{Theorem}[section]

\newtheorem{proposition}{Proposition}[section]

\theoremstyle{definition}

\theoremstyle{remark}
\newtheorem{remark}{Remark}[section]

\theoremstyle{example}
\newtheorem{example}{Example}[section]

\renewcommand{\r}{\rho}
\newcommand{\g}{\gamma}
\renewcommand{\t}{\theta}

\newcommand{\e}{\varepsilon}
\newcommand{\s}{\sigma}

\renewcommand{\div}{\text{div}}
\renewcommand{\O}{\Omega}
\newcommand{\G}{\Gamma}
\renewcommand{\k}{\kappa}

\renewcommand{\l}{\lambda}
\renewcommand{\d}{\partial}
\renewcommand{\a}{\alpha}
\renewcommand{\b}{\beta}

\newcommand{\n}{{\bf n}}
\newcommand{\R}{{\mathbb R}}

\def\be{\begin{equation}}
\def\ee{\end{equation}}
\def\bes{\begin{equation*}}
\def\ees{\end{equation*}}

\def\bc{\begin{cases}}
\def\ec{\end{cases}}
\numberwithin{equation}{section}

\begin{document}

\title[Isometric Immersions and Compensated Compactness]
{Isometric Immersions and Compensated Compactness}
\author{Gui-Qiang Chen \and Marshall Slemrod \and Dehua Wang}
\address{G.-Q. Chen, Department of Mathematics, Northwestern University,
         Evanston, IL 60208.}
\email{\tt gqchen@math.northwestern.edu}

\address{M. Slemrod, Department of Mathematics, University of Wisconsin,
Madison, WI 53706.} \email{\tt slemrod@math.wisc.edu}

\address{D. Wang, Department of Mathematics, University of Pittsburgh,
                Pittsburgh, PA 15260.}
\email{\tt dwang@math.pitt.edu}

\keywords{} \subjclass[2000]{Primary: 53C42, 53C21, 53C45, 58J32,
35L65, 35M10, 35B35; Secondary: 53C24, 57R40, 57R42, 76H05, 76N10}
\date{\today}
\thanks{}

\begin{abstract}
A fundamental problem in differential geometry is to characterize
intrinsic metrics on a two-dimensional Riemannian manifold
${\mathcal M}^2$ which can be realized as isometric immersions into
$\R^3$. This problem can be formulated as initial and/or boundary
value problems for a system of nonlinear partial differential
equations of mixed elliptic-hyperbolic type whose mathematical
theory is largely incomplete. In this paper, we develop a general
approach, which combines a fluid dynamic formulation of balance laws
for the Gauss-Codazzi system with a compensated compactness
framework, to deal with the initial and/or boundary value problems
for isometric immersions in $\R^3$. The compensated compactness
framework formed here is a natural formulation to ensure the weak
continuity of the Gauss-Codazzi system for approximate solutions,
which yields the isometric realization of two-dimensional surfaces
in $\R^3$.

As a first application of this approach, we study the isometric
immersion problem for two-dimensional Riemannian manifolds with
strictly negative Gauss curvature. We prove that
there exists a $C^{1,1}$ isometric immersion of the two-dimensional
manifold in $\R^3$ satisfying our prescribed initial conditions. To
achieve this, we introduce a vanishing viscosity method depending on
the features of initial value problems for isometric immersions and
present a technique to make the apriori estimates including the
$L^\infty$ control and $H^{-1}$--compactness for the viscous
approximate solutions. This yields the weak convergence of the
vanishing viscosity approximate solutions and the weak continuity of
the Gauss-Codazzi system for the approximate solutions, hence the
existence of an isometric immersion of the manifold into $\R^3$
satisfying our initial conditions.
\end{abstract}

\maketitle

\section{Introduction}

A fundamental problem in differential geometry is to characterize
intrinsic metrics on a two-dimensional Riemannian manifold
${\mathcal M}^2$ which can be realized as isometric immersions into
$\R^3$ (cf. Yau \cite{Yau00}; also see \cite{HanHong,PS,Roz}).
Important results have been achieved for the embedding of surfaces
with positive Gauss curvature which can be formulated as an elliptic
boundary value problem (cf. \cite{HanHong}).
For the case of surfaces of negative Gauss curvature where the
underlying partial differential equations are hyperbolic, the
complimentary problem would be an initial or initial-boundary value
problem.
%
%
Hong in \cite{Hong93} first proved that complete negatively curved
surfaces can be isometrically immersed in $\R^3$ if the Gauss
curvature decays at certain rate in the time-like direction.
In fact, a crucial lemma in Hong \cite{Hong93} (also see Lemma
10.2.9 in \cite{HanHong}) shows that, for such a decay rate of the
negative Gauss curvature, there exists a unique global smooth, small
solution forward in time for prescribed smooth, small initial data.
Our main theorem, Theorem 5.1(i), indicates that in fact we can
solve the corresponding problem for a class of {\it large}
non-smooth initial data.
Possible implication of our approach may be in existence theorems
for equilibrium configurations of a catenoidal shell as detailed in
Vaziri-Mahedevan \cite{VM}. When the Gauss curvature changes sign,
the immersion problem then becomes an initial-boundary value problem
of mixed elliptic-hyperbolic type, which is still under
investigation.

The purpose of this paper is to introduce a general approach, which
combines a fluid dynamic formulation of balance laws with a
compensated compactness framework, to deal with the isometric
immersion problem in $\R^3$ (even when the Gauss curvature changes
sign). In Section 2, we formulate the isometric immersion problem
for two-dimensional Riemannian manifolds in $\R^3$ via solvability
of the Gauss-Codazzi system. In Section 3, we introduce a fluid
dynamic formulation of balance laws for the Gauss-Codazzi system for
isometric immersions. Then, in Section 4, we form a compensated
compactness framework and present one of our main observations that
this framework is a natural formulation to ensure the weak
continuity of the Gauss-Codazzi system for approximate solutions,
which yields the isometric realization of two-dimensional surfaces
in $\R^3$.

As a first application of this approach, in Section 5, we focus on
the isometric immersion problem of two-dimensional Riemannian
manifolds with strictly negative Gauss curvature. Since the local
existence of smooth solutions follows from the standard hyperbolic
theory, we are concerned here with the global existence of solutions
of the initial value problem with large initial data.
The metrics $g_{ij}$ we study have special structures and forms
usually associated with

(i) the catenoid of revolution when $g_{11}=g_{22}=cosh(x)$ and
$g_{12}=0$;

(ii) the helicoid when $g_{11}=\lambda^2+y^2$, $g_{22}=1$, and
$g_{12}=0$.

\noindent For these cases, while Hong's theorem \cite{Hong93}
applies to obtain the existence of a solution for small smooth
initial data, our result yields a large-data existence theorem for a
$C^{1,1}$ isometric immersion.

%
%

To achieve this, we introduce a vanishing viscosity method depending
on the features of the initial value problem for isometric
immersions and present a technique to make the apriori estimates
including the $L^\infty$ control and $H^{-1}$--compactness for the
viscous approximate solutions. This yields the weak convergence of
the vanishing viscosity approximate solutions and the weak
continuity of the Gauss-Codazzi system for the approximate
solutions, hence the existence of a $C^{1,1}$--isometric immersion
of the manifold into $\R^3$ with prescribed initial conditions.

We remark in passing that, for the fundamental ideas and early
applications of compensated compactness, see the classical papers by
Tartar \cite{Ta1} and Murat \cite{Mu2}. For applications to the
theory of hyperbolic conservation laws, see for example
\cite{Chen2,Dafermos-book,DiPerna3,evans,Serre}. In particular, the
compensated compactness approach has been applied in
\cite{Chen1,CLF,DCL,DiPerna1,LPS,LPT} to the one-dimensional Euler
equations for unsteady isentropic flow, allowing for cavitation, in
Morawetz \cite{Mor85,Mor95} and Chen-Slemrod-Wang \cite{CSW} for
two-dimensional steady transonic flow away from stagnation points,
and in Chen-Dafermos-Slemrod-Wang \cite{CDSW} for subsonic-sonic
flows.


\section{The Isometric Immersion Problem for Two-Dimensional Riemannian
Manifolds in $\R^3$}

In this section, we formulate the isometric immersion problem for
two-dimensional Riemannian manifolds in $\R^3$ via solvability of
the Gauss-Codazzi system.

Let $\O\subset\R^2$ be an open set. Consider a map ${\bf r}:
\O\to\R^3$ so that, for $(x,y)\in\O$, the two vectors $\{\d_x{\bf
r}, \d_y{\bf r}\}$ in $\R^3$ span the tangent plane at ${\bf
r}(x,y)$ of the surface ${\bf r}(\O)\subset\R^3$. Then
$$
\n=\frac{\d_x{\bf
r}\times\d_y{\bf r}} {|\d_x{\bf r}\times\d_y{\bf r}|}
$$
is the unit normal of the surface ${\bf r}(\O)\subset\R^3$. The
metric on the surface in $\R^3$ is
\begin{equation}\label{2.1}
ds^2=d{\bf r}\cdot d{\bf r}
\end{equation} or, in local
$(x,y)$--coordinates,
\begin{equation}\label{2.2}
ds^2
    =(\d_x{\bf r}\cdot \d_x{\bf r})\,(dx)^2
     +2(\d_x{\bf r}\cdot\d_y{\bf r})\, dxdy
     +(\d_y{\bf r}\cdot \d_y{\bf r})\, (dy)^2.
\end{equation}

Let $g_{ij}, i, j=1,2,$ be the given metric of a two-dimensional
Riemannian manifold $\mathcal{M}$ parameterized on $\Omega$. The
first fundamental form $I$ for $\mathcal{M}$ on $\Omega$ is
\begin{equation}\label{2.3}
I:=g_{11}(dx)^2+2 g_{12}dxdy +g_{22}(dy)^2.
\end{equation}
Then {\em the isometric immersion problem} is to {\it seek a map
$\mathbf{r}: \Omega\to \R^3$ such that
$$
d\mathbf{r}\cdot d\mathbf{r}=I,
$$
that is,
\begin{equation} \label{2.4}
\partial_x{\bf r}\cdot\partial_x{\bf r}=g_{11},\quad
    \partial_x{\bf r}\cdot\partial_y{\bf r}=g_{12},\quad
   \partial_y{\bf r}\cdot\partial_y{\bf r}=g_{22},
\end{equation}
so that $\{\partial_x{\bf r}, \partial_y{\bf r}\}$ in $\R^3$ are
linearly independent}.

The equations in \eqref{2.4} are three nonlinear partial
differential equations for the three components of ${\bf r}(x,y)$.
%

The corresponding second fundamental form is
\begin{equation}\label{2.5}
I\!I:=-d\n\cdot d{\bf r}=h_{11}(dx)^2+ 2h_{12}dxdy + h_{22}(dy)^2,
\end{equation}
and $(h_{ij})_{1\le i,j\le 2}$ is the orthogonality of $\n$ to the
tangent plane. Since $\n\cdot d{\bf r}=0$, then $d(\n\cdot d{\bf
r})=0$ implies
$$
-I\!I+\n\cdot d^2{\bf r}=0, \quad i.e., \quad
  I\!I=(\n\cdot\partial_x^2{\bf r})\, (dx)^2
+2(\n\cdot \partial_{xy}^2{\bf r})\, dxdy+ (\n\cdot\partial_y^2{\bf
r})\, (dy)^2.
$$


The fundamental theorem of surface theory (cf.
\cite{doC1992,HanHong}) indicates that {\it there exists a surface
in $\R^3$ whose first and second fundamental forms are $I$ and
$I\!I$ if the coefficients $(g_{ij})$ and $(h_{ij})$ of the two
given quadratic forms $I$ and $I\!I$ with $(g_{ij})>0$
satisfy the Gauss-Codazzi system}. It is indicated in Mardare
\cite{Mardare2} (Theorem 9; also see \cite{Mardare1}) that this
theorem holds even when $(h_{ij})$ is only in $L^\infty$ for given
$(g_{ij})$ in $C^{1,1}$, for which the immersion surface is
$C^{1,1}$. This shows that, for the realization of a two-dimensional
Riemannian manifold in $\R^3$ with given metric $(g_{ij})>0$, it
suffices to solve $(h_{ij})\in L^\infty$ determined by the
Gauss-Codazzi system to recover ${\bf r}$ a posteriori.


The simplest way to write the Gauss-Codazzi system (cf.
\cite{doC1992,HanHong}) is as
\begin{equation} \label{g1}
\begin{split}
\d_x{M}-\d_y{L}&=\G^{(2)}_{22}L-2\G^{(2)}_{12}M+\G^{(2)}_{11}N, \\
\d_x{N}-\d_y{M}&=-\G^{(1)}_{22}L+2\G^{(1)}_{12}M-\G^{(1)}_{11}N,
\end{split}
\end{equation}
with
\begin{equation}\label{g2}
LN-M^2=\k.
\end{equation}
Here
$$
L=\frac{h_{11}}{\sqrt{|g|}}, \qquad M=\frac{h_{12}}{\sqrt{|g|}},
\qquad N=\frac{h_{22}}{\sqrt{|g|}},
$$
$|g|=det(g_{ij})=g_{11}g_{22}-g_{12}^2$, $\kappa(x,y)$ is the Gauss
curvature that is determined by the relation:
$$
\k(x,y)=\frac{R_{1212}}{|g|}, \qquad
R_{ijkl}=g_{lm}\left(\d_k\G^{(m)}_{ij}-\d_j\G^{(m)}_{ik}
+\G^{(n)}_{ij}\G^{(m)}_{nk}-\G^{(n)}_{ik}\G^{(m)}_{nj}\right),
$$
$R_{ijkl}$ is the curvature tensor and depends on $(g_{ij})$ and its
first and second derivatives, and
$$
\G_{ij}^{(k)}=\frac12g^{kl}\left(\d_j g_{il}+\d_i g_{jl}-\d_l
 g_{ij}\right)
$$
is the Christoffel symbol and depends on the first derivatives of
$(g_{ij})$, where the summation convention is used, $(g^{kl})$
denotes the inverse of $(g_{ij})$, and
$(\partial_1,\partial_2)=(\partial_x, \partial_y)$.

Therefore, given a positive definite metric $(g_{ij})\in C^{1,1}$,
the Gauss-Codazzi system gives us three equations for the three
unknowns $(L, M, N)$ determining the second fundamental form $I\!I$.
Note that, although $(g_{ij})$ is positive definite, $R_{1212}$ may
change sign and so does the Gauss curvature $\k$. Thus, as we will
discuss in Section 3, the Gauss-Codazzi system
\eqref{g1}--\eqref{g2} generically is of mixed hyperbolic-elliptic
type, as in transonic flow (cf.
\cite{Bers,CSW,CourantFriedrichs,Morawetz}). In \S 3--4, we
introduce a general approach to deal with the isometric immersion
problem involving nonlinear partial differential equations of mixed
hyperbolic-elliptic type by combining a fluid dynamic formulation of
balance laws in \S 3 with a compensated compactness framework in \S
4. As an example of direct applications of this approach, in \S 5,
we show how this approach can be applied to establish an isometric
immersion of a two-dimensional Riemannian manifold with negative
Gauss curvature in $\R^3$.

\section{Fluid Dynamic Formulation for the Gauss-Codazzi System}

From the viewpoint of geometry, the constraint condition \eqref{g2}
is a Monge-Amp\`{e}re equation and the equations in \eqref{g1} are
integrability relations. However, our goal here is to put the
problem into a fluid dynamic formulation so that the isometric
immersion problem may be solved via the approaches that have shown
to be useful in fluid dynamics for solving nonlinear systems of
balance laws.
To achieve this, we
formulate the isometric immersion problem via solvability of the
Gauss-Codazzi system \eqref{g1} under constraint \eqref{g2}, that
is, solving {\it first} for $h_{ij}, i,j=1,2,$ via \eqref{g1} with
constraint \eqref{g2} and {\it then} recovering ${\bf r}$ a
posteriori.

To do this, we set
$$
L=\r v^2+p,  \quad M=-\r uv, \quad  N=\r u^2+p,
$$
and  set $q^2=u^2+v^2$ as usual. Then the equations in \eqref{g1}
become the familiar balance laws of momentum:
\begin{equation} \label{g3}
\begin{split}
&\d_x(\r uv)+\d_y(\r v^2+p)
 =-(\r v^2+p)\G^{(2)}_{22}-2\r uv\G^{(2)}_{12}-(\r u^2+p)\G^{(2)}_{11}, \\
&\d_x(\r u^2+p)+\d_y(\r uv)
 =-(\r v^2+p)\G^{(1)}_{22}-2\r uv\G^{(1)}_{12}-(\r u^2+p)\G^{(1)}_{11},
\end{split}
\end{equation}
and the Monge-Amp\`{e}re constraint \eqref{g2} becomes
\begin{equation}\label{g4}
\r p q^2+p^2=\k.
\end{equation}
From this, we can see that, if the Gauss curvature $\k$ is allowed
to be both positive and negative, the ``pressure" $p$ cannot be
restricted to be positive. Our simple choice for $p$ is the
Chaplygin-type gas:
\begin{equation*}\label{g5}
p=-\frac{1}{\r}.
\end{equation*}
Then, from \eqref{g4}, we find
$$
-q^2+\frac1{\r^2}=\k,
$$
and hence we have the ``Bernoulli" relation:
\begin{equation}\label{g6}
\r=\frac{1}{\sqrt{q^2+\k}}.
\end{equation}
This yields
\begin{equation}\label{g7}
p=-\sqrt{q^2+\k},
\end{equation}
and the formulas for $u^2$ and $v^2$:
$$
u^2=p(p-M), \qquad v^2=p(p-L), \qquad M^2=(N-p)(L-p).
$$
The last relation for $M^2$ gives the relation for $p$ in terms of
$(L,M,N)$, and then the first two give the relations for $(u, v)$ in
terms of $(L,M,N)$.

We rewrite \eqref{g3} as
\begin{equation} \label{g8}
\begin{split}
&\d_x(\r uv)+\d_y(\r v^2+p) =R_1, \\
&\d_x(\r u^2+p)+\d_y(\r uv) =R_2,
\end{split}
\end{equation}
where $R_1$ and $R_2$ denote the right-hand sides of \eqref{g3}.

We now find the corresponding ``geometric rotationality--continuity
equations". Multiplying the first equation of \eqref{g8} by $v$ and
the second by $u$, and setting
$$
\d_x v-\d_y u=-\s,
$$
we see
\begin{equation*}
\begin{split}
&\frac{v}\r\div(\r u,\r v)-\frac12 \d_y\k=\frac{R_1}\r+\s u,\\
&\frac{u}\r\div(\r u,\r v)-\frac12 \d_x\k=\frac{R_2}\r-\s v,
\end{split}
\end{equation*}
and hence
\begin{equation} \label{g9}
\begin{split}
&\div(\r u,\r v) =\frac12\frac{\r}{v} \d_y\k
 +\frac{R_1}{v}+\frac{\r u\s}{v},\\
&\div(\r u,\r v) =\frac12\frac{\r}{u}\d_x\k
 +\frac{R_2}{u}-\frac{\r v\s}{u}.
\end{split}
\end{equation}
Thus, the right hand sides of \eqref{g9} are equal, which gives a
formula for $\s$:
\begin{equation} \label{g10}
\s=\frac1{\r q^2}\Big(v\big(\frac12\r \d_x\k+R_2\big)
 -u\big(\frac12\r \d_y\k+R_1\big)\Big).
\end{equation}
If we substitute this formula for $\s$ into \eqref{g9}, we can write
down our ``rotationality-continuity equations" as
\begin{gather}
\d_x v- \d_y u =\frac1{\r q^2}\Big(u\big(\frac12\r \d_y\k+R_1\big)-
v\big(\frac12\r \d_x\k+R_2\big)
         \Big), \label{g11} \\
\d_x(\r u)+ \d_y(\r v)= \frac12\frac{\r u}{q^2} \d_x\k
+\frac12\frac{\r v}{q^2} \d_y\k
 +\frac{v}{q^2}R_1+\frac{u}{q^2}R_2. \label{g12}
\end{gather}
In summary, the Gauss-Codazzi system \eqref{g1}--\eqref{g2}, the
momentum equations \eqref{g3}--\eqref{g7},
and the rotationality-continuity equations \eqref{g6} and
\eqref{g11}--\eqref{g12} are all formally equivalent.

\smallskip
However, for weak solutions, we know from our experience with gas
dynamics that this equivalence breaks down. In
Chen-Dafermos-Slemrod-Wang \cite{CDSW}, the decision was made (as is
standard in gas dynamics) to solve the rotationality-continuity
equations and view the momentum equations as ``entropy" equalities
which may become inequalities for weak solutions. In geometry, this
situation is just the reverse. It is the Gauss-Codazzi system that
must be solved exactly and hence the rotationality-continuity
equations will become ``entropy" inequalities for weak solutions.

The above issue becomes apparent when we set up ``viscous"
regularization that preserves the ``divergence" form of the
equations, which will be introduced in \S 5.3. This is crucial since
we need to solve \eqref{g11}--\eqref{g12} exactly, as we have noted.

\smallskip
To continue further our analogy, let us define the ``sound" speed:
\begin{equation}\label{g18}
c^2=p'(\r),
\end{equation}
which in our case gives
\begin{equation}\label{g19}
c^2=\frac1{\r^2}.
\end{equation}
Since our ``Bernoulli" relation is \eqref{g6}, we see
\begin{equation}\label{g20}
 c^2=q^2+\k.
\end{equation}
Hence, under this formulation,

(i) when $\k>0$, the ``flow" is subsonic, i.e., $q<c$,
       and system \eqref{g3}--\eqref{g4} is elliptic;

(ii) when $\k<0$, the ``flow" is supersonic, i.e., $q>c$,
       and system \eqref{g3}--\eqref{g4} is hyperbolic;

(iii) when $\k=0$, the ``flow" is sonic, i.e., $q=c$,
       and system \eqref{g3}--\eqref{g4} is degenerate.

In general, system \eqref{g3}--\eqref{g4} is of mixed
hyperbolic-elliptic type. Thus, the isometric immersion problem
involves the existence of solutions to nonlinear partial
differential equations of mixed hyperbolic-elliptic type.

\section{Compensated Compactness Framework for Isometric Immersions}

In this section, we form a compensated compactness framework and
present our new observation that this framework is a natural
formulation to ensure the weak continuity of the Gauss-Codazzi
system for approximate solutions, which yields the isometric
realization of two-dimensional Riemannian manifolds in $\R^3$.

\medskip
Let a sequence of functions $(L^\varepsilon, M^\varepsilon,
N^\varepsilon)(x,y)$, defined on an open subset $\Omega\subset
\R^2$, satisfy the following Framework (A):

\medskip
{\rm (A.1)}\, $|(L^\varepsilon, M^\varepsilon,
N^\varepsilon)(x,y)|\le C$\,\, a.e.\, $(x,y)\in \Omega$, for some
$C>0$ independent of $\e$;

\smallskip
{\rm (A.2)}\,  $\partial_xM^\varepsilon-\partial_y L^\varepsilon$
and $\partial_xN^\varepsilon-\partial_yM^\varepsilon$ are confined
in a compact set in $H_{loc}^{-1}(\Omega)$;

\smallskip
{\rm (A.3)\, There exist $o^\varepsilon_j(1), j=1,2,3$, with
$o^\varepsilon_j(1)\to 0$ in the sense of distributions as
$\varepsilon\to 0$ such that
\begin{equation} \label{g1-b}
\begin{split}
\d_x{M^\e}-\d_y{L^\e}
&=\G^{(2)}_{22}L^\e-2\G^{(2)}_{12}M^\e+\G^{(2)}_{11}N^\e +o^\e_1(1), \\
\d_x{N^\e}-\d_y{M^\e}
&=-\G^{(1)}_{22}L^\e+2\G^{(1)}_{12}M^\e-\G^{(1)}_{11}N^\e+o^\e_2(1),
\end{split}
\end{equation}
and
\begin{equation}\label{g2-b}
L^\varepsilon N^\varepsilon-(M^\varepsilon)^2=\kappa+o^\e_3(1).
\end{equation}

\medskip

Then we have

\begin{theorem}[Compensated compactness framework] \label{T}
$\quad$ Let a sequence of functions $(L^\varepsilon, M^\varepsilon,
N^\varepsilon)(x,y)$ satisfy Framework {\rm (A)}. Then there exists
a subsequence (still labeled) $(L^\varepsilon, M^\varepsilon,
N^\varepsilon)(x,y)$ that converges weak-star in $L^\infty(\Omega)$
to $(\bar{L}, \bar{M}, \bar{N})$ as $\varepsilon\to 0$ such that

\begin{enumerate}
\renewcommand{\theenumi}{\roman{enumi}}
\item $|(\bar{L}, \bar{M}, \bar{N})(x,y)|\le C \qquad
\mbox{a.e.} \,\, (x,y)\in\Omega$;

\item the Monge-Amp\'{e}re constraint \eqref{g2} is weakly
continuous with respect to the subsequence $(L^\varepsilon,
M^\varepsilon, N^\varepsilon)(x,y)$ that converges weak-star in
$L^\infty(\Omega)$  to $(\bar{L}, \bar{M}, \bar{N})$ as
$\varepsilon\to 0$;

\item the Gauss-Codazzi equations in \eqref{g1} hold.
\end{enumerate}

\noindent That is, the limit $(\bar{L}, \bar{M}, \bar{N})$ is a
bounded weak solution to the Gauss-Codazzi system
\eqref{g1}--\eqref{g2}, which yields an isometric realization of the
corresponding two-dimensional Riemannian manifold in $\R^3$.
\end{theorem}

\smallskip
\begin{proof}
By the div-curl lemma of Tartar-Murat \cite{Ta1,Mu2} and the Young
measure representation theorem for a uniformly bounded sequence of
functions (cf. Tartar \cite{Ta1}), we employ (A.1)--(A.2) to
conclude that there exist a family of Young measures
$\{\nu_{x,y}\}_{(x,y)\in\Omega}$ and a subsequence (still labeled)
$(L^\varepsilon, M^\varepsilon, N^\varepsilon)(x,y)$ that converges
weak-star in $L^\infty(\Omega)$ to $(\bar{L}, \bar{M}, \bar{N})$ as
$\varepsilon\to 0$ such that

\medskip
(a)\, $(\bar{L}, \bar{M}, \bar{N})(x,y)=(\langle
\nu_{x,y},\,L\rangle, \langle\nu_{x,y},\, M\rangle,
\langle\nu_{x,y},\, N\rangle) \,\, \quad {\rm a.e.}\,\,
(x,y)\in\Omega$;

(b) \, $|(\bar{L}, \bar{M}, \bar{N})(x,y)|\le C \qquad \mbox{a.e.}
\,\, (x,y)\in\Omega$;

(c)\,  the following commutation identity holds:
\begin{equation}\label{4.1-a}
\langle\nu_{x,y}, M^2-LN\rangle=\langle\nu_{x,y},
M\rangle^2-\langle\nu_{x,y}, L\rangle\langle\nu_{x,y}, N\rangle=
(\bar{M})^2-\bar{L}\bar{N}.
\end{equation}

Since the equations in \eqref{g1-b} are linear in
$(L^\e,M^\e,N^\e)$, then the limit $(\bar{L}, \bar{M}, \bar{N})$
also satisfies the equations in \eqref{g1} in the sense of
distributions.

Furthermore, condition \eqref{g2-b} yields that
\begin{equation}\label{4.2-a}
\langle \nu_{x,y}, LN-M^2\rangle=\k(x,y) \qquad {\rm a.e.}\,\,
(x,y)\in\Omega.
\end{equation}
The combination \eqref{4.1-a} with \eqref{4.2-a} yields the weak
continuity of the Monge-Amp\'{e}re constraint with respect to the
sequence $(L^\e, M^\e, N^\e)$ that converges weak-star in
$L^\infty(\Omega)$ to $(\bar{L}, \bar{M}, \bar{N})$ as $\e\to 0$:
$$
\bar{L}\bar{N}-(\bar{M})^2=\k.
$$
Therefore, $(\bar{L},\bar{M},\bar{N})$ is a bounded weak solution of
the Gauss-Codazzi system \eqref{g1}--\eqref{g2}. Then the
fundamental theorem of surface theory implies an isometric
realization of the corresponding two-dimensional Riemannian manifold
in $\R^3$. This completes the proof.
\end{proof}

\begin{remark} In the compensated compactness framework, Condition
(A.1) can be relaxed to the following condition:

\smallskip
\mbox{(A.1)'} \, $\|(L^\e, M^\e, N^\e)\|_{L^p(\Omega)} \le C, \qquad
p>2$, for some $C>0$ independent of $\e$.

\smallskip
\noindent Then all the arguments for Theorem 4.1 follow only with
the weak convergence in $L^p(\Omega), p>2$, replacing the weak-star
convergence in $L^\infty(\Omega)$, with the aid of the Young measure
representation theorem for a uniformly $L^p$ bounded sequence of
functions (cf. Ball \cite{Ball}).
\end{remark}

\smallskip
There are various ways to construct approximate solutions by either
analytical methods, such as vanishing viscosity methods and
relaxation methods, or numerical methods, such as finite difference
schemes and finite element methods.  Even though the solution to the
Gauss-Codazzi system  may eventually turn out to be more regular,
especially in the region of  positive Gauss curvature $\kappa>0$,
the point of considering weak solutions here is to demonstrate that
such solutions may be constructed by merely using very crude
estimates. Such estimates are available in a variety of
approximating methods through basic energy-type estimates, besides
the $L^\infty$ estimate. On the other hand, in the region of
negative Gauss curvature $\kappa<0$, discontinuous solutions are
expected so that the estimates can be improved at most up to $BV$ in
general.

\smallskip
The compensated compactness framework (Theorem 4.1) indicates that,
in order to find an isometric immersion, it suffices to construct a
sequence of approximate solutions $(L^\varepsilon, M^\varepsilon,
N^\varepsilon)(x,y)$ satisfying Framework (A), which yields its weak
limit $(\bar{L},\bar{M},\bar{N})$ to be an isometric immersion. To
achieve this through the fluid dynamic formulation \eqref{g3} and
\eqref{g6} (or \eqref{g7}), it requires a uniform $L^\infty$
estimate of $(u^\e, v^\e)$ such that the sequence
$$
(L^\e, M^\e, N^\e)=(\r^\e (v^\e)^2+p^\e, -\r^\e u^\e v^\e, \r^\e
(u^\e)^2+p^\e)
$$
with
$$
p^\e=-\frac{1}{\r^\e}=\sqrt{(u^\e)^2+(v^\e)^2+\kappa}
$$
satisfies Framework (A).

\smallskip
The fluid dynamic formulation, \eqref{g3} and \eqref{g6} (or
\eqref{g7}), and the compensated compactness framework (Theorem 4.1)
provide a unified approach to deal with various isometric immersion
problems even for the case when the Gauss curvature changes sign,
that is, for the equations of mixed elliptic-hyperbolic type.

\section{Isometric Immersions of Two-Dimensional Riemannian Manifolds
 with Negative Gauss Curvature}

As a first example, in this section, we show how this approach can
be applied to establish an isometric immersion of a two-dimensional
Riemannian manifold with negative Gauss curvature in $\R^3$.

\subsection{Reformulation}

In this case, $\kappa<0$ in $\Omega$ and, more specifically,
$$
\kappa=-\gamma^2, \quad \gamma> 0
\qquad\,\, \text{in}\,\, \Omega.
$$
%
%
For convenience, we assume $\gamma\in C^1$ in this section and
rescale $(L,M,N)$ in this case as
$$
\tilde{L}=\frac{L}{\g}, \qquad \tilde{M}=\frac{M}{\g}, \qquad
\tilde{N}=\frac{N}{\g},
$$
so that \eqref{g2} becomes
$$
\tilde{L}\tilde{N}-\tilde{M}^2=-1.
$$
Then, without ambiguity, we redefine the ``fluid variables" via
$$
\tilde{L}=\r v^2+p, \qquad \tilde{M}=-\r uv, \qquad  \tilde{N}=\r
u^2+p,
$$
and set $q^2=u^2+v^2$, where we have still used $(u,v,p,\rho)$ as
the scaled variables and will use them hereafter (although they are
different from those in \S 2--\S 4).

Then the equations in \eqref{g1} become the same form of balance
laws of momentum:
\begin{equation} \label{g3a}
\begin{split}
&\d_x(\r uv)+\d_y(\r v^2+p)
 =R_1,\\
&\d_x(\r u^2+p)+\d_y(\r uv)
 =R_2,
\end{split}
\end{equation}
where
\begin{eqnarray}
&&R_1:= -(\r v^2+p)\tilde{\G}^{(2)}_{22}-2\r uv\tilde{\G}^{(2)}_{12}
-(\r u^2+p)\tilde{\G}^{(2)}_{11}, \label{g3a-1} \\
&&R_2:=-(\r v^2+p)\tilde{\G}^{(1)}_{22}-2\r
uv\tilde{\G}^{(1)}_{12}-(\r u^2+p)\tilde{\G}^{(1)}_{11},
\label{g3a-2}
\end{eqnarray}
$$
\tilde{\G}^{(1)}_{11}={\G}^{(1)}_{11}+\frac{\g_x}{\g}, \qquad
\tilde{\G}^{(1)}_{12}={\G}^{(1)}_{12}+\frac{\g_y}{2\g},\qquad
\tilde{\G}^{(1)}_{22}={\G}^{(1)}_{22},
$$
$$
\tilde{\G}^{(2)}_{11}={\G}^{(2)}_{11}, \qquad
\tilde{\G}^{(2)}_{12}={\G}^{(2)}_{12}+\frac{\g_x}{2\g},\qquad
\tilde{\G}^{(2)}_{22}={\G}^{(2)}_{22}+\frac{\g_y}{\g}.
$$

Furthermore, the constraint $\tilde{L}\tilde{N}-\tilde{M}^2=-1$
becomes
\begin{equation}\label{g4a}
\r p q^2+p^2=-1.
\end{equation}
{}From $p=-\frac1{\r}$ and \eqref{g4a},
we have the ``Bernoulli" relation:
\begin{equation}\label{g6a}
\r=\frac{1}{\sqrt{q^2-1}} \quad\text{or}\quad p=-\sqrt{q^2-1},
\end{equation}
which yields
\begin{equation}\label{g7a}
 u^2=p(p-\tilde{N}), \quad
  v^2=p(p-\tilde{L}), \quad
  (\tilde{M})^2=(\tilde{N}-p)(\tilde{L}-p).
\end{equation}
Then the last relation in \eqref{g7a} gives the relation for $p$ in
terms of $(\tilde{L}, \tilde{M}, \tilde{N})$, and the first two give
the relations for $(u, v)$ in terms of $(\tilde{L}, \tilde{M},
\tilde{N})$.

Similarly to the calculation in \S 3, we can write down our
``rotationality--continuity equations" as
\begin{gather}
\d_xv-\d_y u=-\frac1{\r q^2}\left(v R_2 - u R_1\right)=:S_1, \label{g11a} \\
\d_x(\r u)+ \d_y(\r v) =\frac{v}{q^2}R_1+\frac{u}{q^2}R_2=:S_2.
\label{g12a}
\end{gather}

Under the new scaling, the ``sound" speed is
\begin{equation}\label{g18a}
 c^2=p'(\r)=\frac{1}{\rho^2}>0.
 \end{equation}
Then the ``Bernoulli" relation \eqref{g6} yields
 \begin{equation}\label{g20a}
 c^2=q^2-1.
 \end{equation}
Therefore, $q>c$, and the ``flow" is always supersonic, i.e., the
system is purely hyperbolic.

\subsection{Riemann invariants}

In polar coordinates $(u, v)=(q\cos\t, q\sin\t)$, we have
$$
R_1=\r q^2\cos^2\t \;\tilde{\G}^{(2)}_{22}-2\r
q^2\sin\t\cos\t\;\tilde{\G}^{(2)}_{12} +\r
q^2\sin^2\t\;\tilde{\G}^{(2)}_{11}
-\r\big(\tilde{\G}^{(2)}_{22}+\tilde{\G}^{(2)}_{11}\big),
$$
$$
R_2=\r q^2\cos^2\t \;\tilde{\G}^{(1)}_{22}-2\r
q^2\sin\t\cos\t\;\tilde{\G}^{(1)}_{12} +\r
q^2\sin^2\t\;\tilde{\G}^{(1)}_{11}
-\r\big(\tilde{\G}^{(1)}_{22}+\tilde{\G}^{(1)}_{11}\big),
$$
and then \eqref{g11a} and \eqref{g12a} become
\begin{eqnarray}
&& \sin\t \d_x q +q\cos\t \; \d_x\t-\cos\t\, \d_y q +q\sin\t \; \d_y\t=S_1, \label{g26}\\
&& \frac{\cos\t}{q(q^2-1)}\d_xq +\sin\t \; \d_x\t
 +\frac{\sin\t}{q(q^2-1)}\d_y q -\cos\t \; \d_y\t = -\frac{\sqrt{q^2-1}}{q}S_2.
 \label{g27}
\end{eqnarray}
That is, as a first-order system, \eqref{g11a} and \eqref{g12a} can
be written as
\begin{equation} \label{g28}
 \begin{bmatrix}
  \sin\t  & q\cos\t \\
  \frac{1}{q(q^2-1)}\cos\t &  \sin\t
 \end{bmatrix}
 \d_x\begin{bmatrix}
 q \\ \t
 \end{bmatrix}   +
 \begin{bmatrix}
  -\cos\t  & q\sin\t \\
  \frac{1}{q(q^2-1)}\sin\t &  -\cos\t
 \end{bmatrix}
 \d_y\begin{bmatrix}
 q \\ \t
 \end{bmatrix}  =
 \begin{bmatrix}
 S_1 \\  -\frac{\sqrt{q^2-1}}{q}S_2
 \end{bmatrix}.
\end{equation}
One of our main observations is that, under this reformation, the
two coefficient matrices in \eqref{g28} actually commute,
which guarantees that they have common eigenvectors. The eigenvalues
of the first and second matrices are
$$
\l_\pm=\sin\t\pm \frac{\cos\t}{\sqrt{q^2-1}}, \qquad
 \mu_\pm=-\cos\t\pm \frac{\sin\t}{\sqrt{q^2-1}},
$$
and the common left eigenvectors of the two coefficient matrices are
$$
(\pm\frac{1}{q\sqrt{q^2-1}}, 1).
$$
Thus, we may define the Riemann invariants $W_\pm=W_\pm(\theta, q)$
as
\begin{equation} \label{r1}
 \d_\t W_\pm=1, \quad
 \d_q W_\pm
 =\pm\frac{1}{q\sqrt{q^2-1}},
\end{equation}
which yields
\begin{equation} \label{r2}
 W_\pm=\t\pm\arccos\big(\frac{1}q\big).
\end{equation}

Now multiplication \eqref{g28}  by $(\d_q W_\pm, \d_\t W_\pm)$ from
the left yields
 \begin{gather}
 \l_+\big(\d_q W_+ \, \d_xq+ \d_\t W_+\, \d_x\t\big)
 +\mu_+\big(\d_q W_+\, \d_yq + \d_\t W_+\, \d_y\t\big)
 =S_1 \d_q W_+ -\frac{\sqrt{q^2-1}}{q}S_2,  \label{g29} \\
 \l_-\big(\d_q W_-\, \d_xq+ \d_\t W_-\, \d_x\t\big)
 +\mu_-\big( \d_q W_-\, \d_yq + \d_\t W_-\, \d_y\t\big)
 =S_1 \d_q W_- -\frac{\sqrt{q^2-1}}{q}S_2. \label{g30}
 \end{gather}
From \eqref{r2},
$$
\d_x W_\pm=\d_q W_+ \,\d_x q + \d_\t W_+\,\d_x\t, \qquad \d_y
W_\pm=\d_q W_+\, \d_yq+ \d_\t W_+\, \d_y\t,
$$
then we can write \eqref{g29} and \eqref{g30} as
\begin{gather}
\l_+ \d_x W_+ +\mu_+ \d_y W_+ =\frac{1}{q\sqrt{q^2-1}}S_1
-\frac{\sqrt{q^2-1}}{q} S_2, \label{g35}\\
\l_-\d_x W_- +\mu_- \d_y W_- =-\frac{1}{q\sqrt{q^2-1}}S_1
-\frac{\sqrt{q^2-1}}{q} S_2. \label{g36}
 \end{gather}

\subsection{Vanishing viscosity method via parabolic regularization}


Now we introduce a vanishing viscosity method via parabolic
regularization to obtain the uniform $L^\infty$ estimate by
identifying invariant regions for the approximate solutions.

First, if $\tilde{R_1}$ and $\tilde{R_2}$ denote the additional
terms that should be added to the right-hand side of the
Gauss-Codazzi system \eqref{g3a}, our first choice is
 \begin{equation} \label{g45b}
 \tilde{R_1}=\e\partial_y^2(\r v), \quad \tilde{R_2}=\e\partial_y^2(\r u),
 \end{equation}
which gives us the system of ``viscous" parabolic regularization:
\begin{equation} \label{g13}
\begin{split}
&\d_x(\r uv)+\d_y(\r v^2+p) =R_1+ \e\partial_y^2(\r v)=R_1+\tilde{R}_1, \\
&\d_x(\r u^2+p)+\d_y(\r uv) =R_2+\e\partial_y^2(\r
u)=R_2+\tilde{R}_2.
\end{split}
\end{equation}
From equations \eqref{g11} and \eqref{g12}, we see
\be \label{g46b}
 \tilde{S_1}=-\frac1{\r q^2}\left(v\tilde{R_2}-u\tilde{R_1}\right), \quad
 \tilde{S_2}=\frac1{q^2}\left(v\tilde{R_1}+u\tilde{R_2}\right)
\ee
should be added to $S_1$ and $S_2$ on the right-hand side of
\eqref{g11} and \eqref{g12}. In polar coordinates
$(u,v)=(q\cos\t,q\sin\t)$, \eqref{g46b} becomes
 \be \label{g47b}
 \tilde{S_1}=\e\frac{2}{\r q}\partial_y\t\partial_y(\r q)+\e\partial_y^2\t,
 \qquad
 \tilde{S_2}=\e\frac{1}{q}\partial_y^2(\r q)-\e\r(\partial_y\t)^2.
  \ee
Note the identity
\bes
\begin{split}
&\e \partial_y^2\Big(\arccos\big(\frac{1}{q}\big)\Big)
   =\e \partial_y^2\left(\text{arccsc} (\r q)\right) \\
&=-\e \partial_y\Big(\frac1{\r
q\sqrt{\r^2q^2-1}}\Big)\,\partial_y(\r q)
  -\frac{\e}{\r^2 q}\partial_y^2(\r q) \\
&=-\e \partial_y\Big(\frac1{\r
q\sqrt{\r^2q^2-1}}\Big)\,\partial_y(\r q)
  -\frac{\tilde{S_2}}{\r^2}-\frac{\e}{\r}(\partial_y\t)^2.
\end{split}
\ees
Then
$$
\frac{\tilde{S_2}}{\r^2}=-\e\partial_y^2\Big(\arccos\big(\frac{1}{q}\big)\Big)
-\e\partial_y\Big(\frac1{\r q\sqrt{\r^2q^2-1}}\Big)\,\partial_y(\r
q) -\frac{\e}{\r}(\partial_y\t)^2,
$$
and
thus
\bes
\begin{split}
&\tilde{S_1}-(q^2-1)\tilde{S_2}=\tilde{S_1}-\frac{\tilde{S_2}}{\r^2}\\
&=\frac{2\e}{\r q}\partial_y\t\,\partial_y(\r q)+\e\partial_y^2\t
+\e\partial_y^2\Big(\arccos\big(\frac{1}{q}\big)\Big)
+\e\partial_y\Big(\frac1{\r q\sqrt{\r^2q^2-1}}\Big)\,\partial_y(\r
q) +\frac{\e}{\r}(\partial_y\t)^2.
\end{split}
\ees
Since
$$
\partial_y\t=\partial_yW_+ +\frac{\partial_y(\r q)}{\r q\sqrt{\r^2 q^2-1}},
$$
then
$$
\tilde{S_1}-(q^2-1)\tilde{S_2}= \e\partial_y^2W_+ +\frac{2\e
q}{\r}\partial_y W_+\,\partial_y(\r q) +\e (\partial_y W_+)^2.
$$
Similarly, using
$$
\partial_y\t=\partial_y W_- -\frac{\partial_y(\r q)}{\r q\sqrt{\r^2 q^2-1}},
$$
we have
$$
-\tilde{S_1}-(q^2-1)\tilde{S_2}= -\e \partial_y^2 W_- -\frac{2\e
q}{\r}\partial_y W_-\,\partial_y(\r q) +\e (\partial_y W_-)^2.
$$

Thus, if we add the above $\tilde{S_1}$ and $\tilde{S_2}$ to the
original $S_1$ and $S_2$, \eqref{g35} and \eqref{g36} become
 \begin{equation} \label{g45}
 \begin{split}
& q\sqrt{q^2-1}
  \Big(\l_+\frac{\d W_+}{\d x}+\mu_+\frac{\d W_+}{\d y}\Big),\\
&\quad = \e \partial_y^2 W_+
 +\frac{2\e q}{\r}\partial_y W_+\,\partial_y(\r q)
+\e (\partial_yW_+)^2+S_1-\left(q^2-1\right) S_2
 \end{split}
\end{equation}
\begin{equation} \label{g46}
\begin{split}
&q\sqrt{q^2-1}
\Big(\l_-\frac{\d W_-}{\d x}+\mu_-\frac{\d W_-}{\d y}\Big)\\
&\quad =-\e \partial_y^2 W_- -\frac{2\e q}{\r}\partial_y
W_-\,\partial_y(\r q) +\e (\partial_y W_-)^2- S_1-\left(q^2-1\right)
S_2.
\end{split}
\end{equation}
Plugging $R_1$ and $R_2$ into $S_1$ and $S_2$ yields \be \label{s1}
\begin{split}
& S_1\pm\left(q^2-1\right) S_2 \\
&= -q
\sin\t\Big(\tilde{\G}_{22}^{(1)}\cos^2\t-2\tilde{\G}_{12}^{(1)}\sin\t\cos\t
       +\tilde{\G}_{11}^{(1)}\sin^2\t-\frac{1}{q^2}\big(\tilde{\G}_{22}^{(1)}
       +\tilde{\G}_{11}^{(1)}\big)\Big)\\
&\quad
-q\cos\t\Big(-\tilde{\G}_{22}^{(2)}\cos^2\t+2\tilde{\G}_{12}^{(2)}\sin\t\cos\t
       -\tilde{\G}_{11}^{(2)}\sin^2\t+\frac{1}{q^2}\big(\tilde{\G}_{22}^{(2)}
       +\tilde{\G}_{11}^{(2)}\big)\Big)\\
&\quad \pm\frac{1}{\r}\left\{
q\cos\t\Big(\tilde{\G}_{22}^{(1)}\cos^2\t-2\tilde{\G}_{12}^{(1)}\sin\t\cos\t
       +\tilde{\G}_{11}^{(1)}\sin^2\t-\frac{1}{q^2}\big(\tilde{\G}_{22}^{(1)}
       +\tilde{\G}_{11}^{(1)}\big)\Big)\right.\\
&\qquad\quad\,\, \left.
+q\sin\t\Big(\tilde{\G}_{22}^{(2)}\cos^2\t-2\tilde{\G}_{12}^{(2)}\sin\t\cos\t
  +\tilde{\G}_{11}^{(2)}\sin^2\t-\frac{1}{q^2}\big(\tilde{\G}_{22}^{(2)}
  +\tilde{\G}_{11}^{(2)}\big)\Big)\right\}.
\end{split}
\ee Then system \eqref{g45}--\eqref{g46} is parabolic when
$\lambda_+>0$ and $\lambda_-<0$.

\medskip
Furthermore, setting $(E,F,G)=(g_{11}, g_{12}, g_{22})$, we recall
the following classical identities:
\begin{eqnarray*}
&{\G}_{11}^{(1)}=\frac{GE_x-2FF_x+FE_y}{2(EG-F^2)}, \qquad
  &{\G}_{22}^{(1)}=\frac{2GF_y-GG_x-FG_x}{2(EG-F^2)},\\
&{\G}_{11}^{(2)}=\frac{2EF_x-EE_y-FE_x}{2(EG-F^2)}, \qquad
  &{\G}_{22}^{(2)}=\frac{EG_y-2FF_y+FG_x}{2(EG-F^2)}, \\
&{\G}_{12}^{(1)}=\frac{GE_y-FG_x}{2(EG-F^2)},\qquad\,\,\,\, \qquad
  &{\G}_{12}^{(2)}=\frac{EG_x-FE_y}{2(EG-F^2)},
\end{eqnarray*}
$$
(EG-F^2)^2\k=\det\begin{bmatrix}
                  -\frac12E_{yy}+F_{xy}-\frac12G_{xx} &\frac12E_x & F_x-\frac12F_y\\
                  F_x-\frac12G_x & E & F\\
                  \frac12 G_y &F & G
                  \end{bmatrix}
            -\det\begin{bmatrix}
                  0&\frac12E_y&\frac12G_x\\
                  \frac12E_y&E&F\\
                  \frac12G_x&F&G
                  \end{bmatrix},
$$
and $\g^2=-\k$.

\medskip
\subsection{$L^\infty$--estimate for the viscous approximate
solutions}

Based on the calculation above for the Riemann invariants, we now
introduce an approach to make the $L^\infty$ estimate. First we need
to sketch the graphs of the level sets of $W_\pm$. If
$$
W_\pm=\t\pm\arccos\big(\frac{1}q\big)=C_\pm \qquad \text{for
constants } C_\pm,
$$
then
$$
\frac{d\t}{d q}=\mp\frac{d}{d q}\Big(\arccos\big(\frac{1}q\big)\Big)
=\mp\frac{1}{q\sqrt{q^2-1}} \qquad\,\,\text{on}\,\, W_\pm=C_\pm,
$$
and, as $q\to\infty$,
$$
\frac{d\t}{d q}\to 0, \qquad
\t\to C_\pm\mp\arccos(0)=C_\pm\mp\frac{\pi}2
\qquad\quad\text{on}\quad W_\pm=C_\pm.
$$
See Fig. \ref{f1} for the graphs of the level sets $W_\pm=C_\pm$.

\begin{figure}[h]
\centering
\includegraphics[width=5.5in]{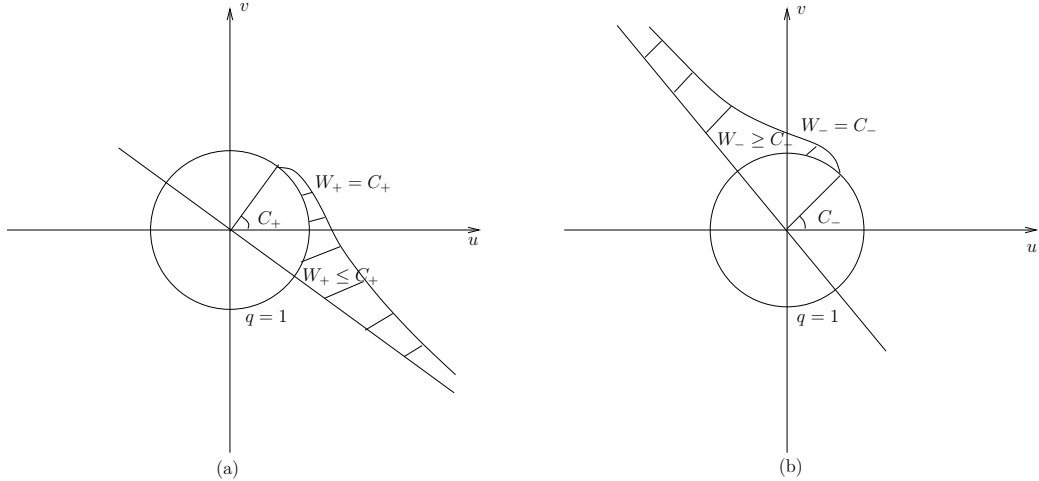}  
\caption{Level sets}
 \label{f1}
\end{figure}

Next we examine the meaning of inequality $W_+\le C_+$, i.e.,
$$
\t+\arccos\big(\frac{1}q\big)\le C_+.
$$
For example, if $q=1$, then $\t\le C_+$. This indicates the region
of $W_+\le C_+$ as sketched in Fig. \ref{f1}(a). Similarly, $W_-\ge
C_-$ means
$$
\t-\arccos\big(\frac{1}q\big)\ge C_-,
$$
and, if $q=1$, then $\t\ge C_-$, and the region of $W_-\ge C_-$ is
sketched in Fig. \ref{f1}(b). Thus we see that
$$
W_+\le C_+ \quad\text{means the region below}\quad W_+=C_+,
$$
$$
W_-\ge C_- \quad\text{means the region below}\quad W_-=C_-;
$$
and
$$
W_+\ge C_+ \quad\text{means the region above}\quad W_+=C_+,
$$
$$
W_-\le C_- \quad\text{means the region above}\quad W_-=C_-.
$$

\bigskip As an example, we now focus on the case that
\begin{equation}\label{5.ex-1}
F=0, \qquad E(x)=G(x).
\end{equation}
Then
$$
{\G}_{11}^{(1)}=\frac{E'}{2E}, \quad {\G}_{12}^{(1)}=0, \quad
{\G}_{22}^{(1)}=-\frac{E'}{2E}; \quad {\G}_{11}^{(2)}=0, \quad
{\G}_{12}^{(2)}=\frac{E'}{2E}, \quad
 {\G}_{22}^{(2)}=0.
$$
Therefore, we have
$$
\tilde{\G}_{11}^{(1)}=\frac{E'}{2E}+\frac{\g'}{\g}, \quad
\tilde{\G}_{12}^{(1)}=0, \quad \tilde{\G}_{22}^{(1)}=-\frac{E'}{2E};
\quad \tilde{\G}_{11}^{(2)}=0, \quad
\tilde{\G}_{12}^{(2)}=\frac{E'}{2E}+\frac{\g'}{2\g}, \quad
\tilde{\G}_{22}^{(2)}=0,
$$
and the right-hand side of \eqref{s1} is equal to
$$
\frac1{2\g^2}\left(\frac{\k'}{\r^2 q}
-\g^2q\frac{E'}{E}\right)\sin\t
\pm\frac{1}{2\g^2\r}\left(-\frac{\k'}{q}
+q\g^2\frac{E'}{E}\right)\cos\t.
$$
Thus, the two solutions $\t_\pm(q)$ that make the right-hand side of
\eqref{s1} equal to zero satisfy
\begin{equation} \label{tan}
\tan\t=\pm\frac{\frac{1}{\r}\left(\frac{\k'}{q} -\g^2 q
\frac{E'}{E}\right)}{\frac{\k'}{\r^2 q}-\g^2q\frac{E'}{E}}.
\end{equation}
If we fix the intersection point of $\t_\pm(q)$ at
$$
\t=0, \quad q=q_0=\beta,
$$
where $\beta>1$ is a constant, then the above ordinary differential
equation \eqref{tan} becomes
\begin{equation}\label{ode-1}
\frac{1}{\beta^2}\frac{\k'(x)}{\k(x)}+\frac{E'(x)}{E(x)}=0,
\end{equation}
i.e.,
$$
\frac{d}{dx}\ln\left(|\k(x)|^{\frac{1}{\beta^2}}E(x)\right)=0.
$$
Thus,
$$
|\k(x)|^{\frac{1}{\beta^2}}E(x)=const. >0.
$$
Since $\k(x)<0$, then
$$
\k(x)=-\k_0 E(x)^{-\beta^2},
$$
where $\k_0>0$ is a constant, and equation \eqref{tan} for $\tan\t$
becomes
$$
\tan\t=\pm\frac{\sqrt{q^2-1}(\beta^2-q^2)}{\beta^2-(\beta^2-1) q^2}.
$$

Assume that we have a solution $E(x)$ to \eqref{ode-1}. Fix another
constant $1<\a<\b$. Then the curves $\t_\pm({q})$ are independent of
$(x,y)$ and look like the sketch in Fig. \ref{f2}.

\begin{figure}[h]
\centering
\includegraphics[width=4.5in]{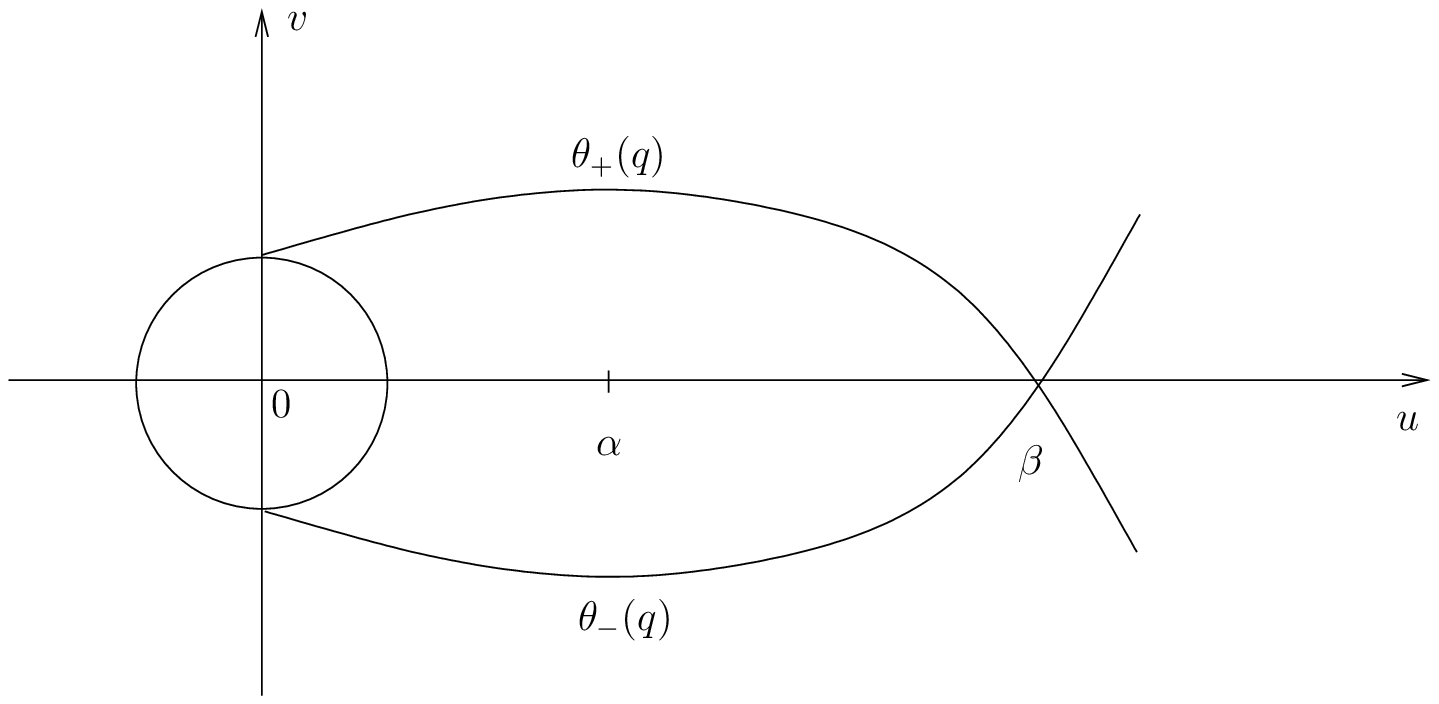}  
\caption{Graphs of $\t_\pm$}
 \label{f2}
\end{figure}

\begin{figure}[h]
\centering
\includegraphics[width=4.5in]{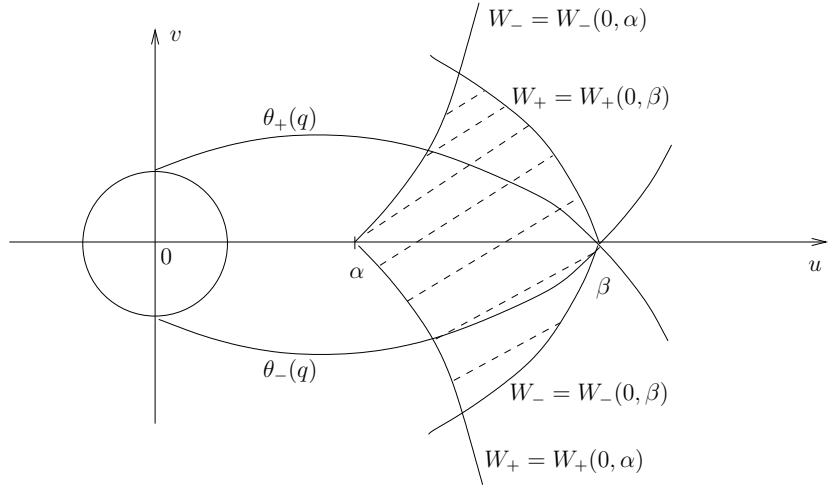}  
\caption{Invariant regions}
 \label{f3}
\end{figure}

Next, note that $W_\pm$, when evaluated at $\t=0, q=\a$, and $\t=0,
q=\b$, take on the constant values, i.e., independent of $(x,y)$.
Equations \eqref{g45} and \eqref{g46} imply that, at any point where
$\nabla W_+=0$ (respectively $\nabla W_-=0$),
$$
\e\g\partial_y^2 W_+ \bc
                >0 \quad &\text{for}\,\, \t>\t_+(q),\\
                <0 \quad &\text{for}\,\, \t<\t_+(q),
                \ec
$$
$$
\e\g\partial_y^2 W_- \bc
                >0 \quad &\text{for}\,\, \t>\t_-(q),\\
                <0,\quad  &\text{for}\,\, \t<\t_-(q).
                \ec
$$
Hence, when $\lambda_+>0$ and $\lambda_-<0$, by the maximum
principle (cf. \cite{evans2,PW}),
$$
W_+ \quad\text{has no internal maximum for}\quad \t>\t_+(q),
$$
$$
W_+ \quad\text{has no internal minimum for}\quad \t<\t_+(q),
$$
$$
W_- \quad\text{has no internal maximum for}\quad \t>\t_-(q),
$$
$$
W_- \quad\text{has no internal minimum for}\quad \t<\t_-(q).
$$
Define
$$
W_\pm(0,\beta)=\pm\cos^{-1}\big(\frac1{\beta}\big),\quad
W_\pm(0,\a)=\pm\cos^{-1}\big(\frac1{\a}\big).
$$
Therefore, if the data is such that $W_+\le W_+(0, {\b})$, then
$W_+$ can have no internal maximum greater than $W_+(0, {\b})$ for
$\t>\t_+(q)$. Similarly, if the data is such that $W_-\ge
W_-(0,{\b})$ ($=W_+(0,{\b})$), then $W_-$ can have no internal
minimum less than $W_-(0,{\b})$ for $\t<\t_-(q)$. Furthermore, if
the data is such that $W_+\ge W_+(0, {\a})$, then $W_+$ can have no
internal minimum less than $W_+(0, {\a})$ for $\t<\t_+(q)$;  if that
data is such that $W_-\le W_-(0, {\a})$, then $W_-$ can have no
internal maximum greater than $W_-(0, {\a})$ for $\t>\t_-(q)$. Thus,
the diamond-shaped region in Fig. \ref{f3}  provides the upper and
lower bounds for $W_\pm$.

From the definition of $\l_\pm$:
$$
\l_\pm
=\sin\t\pm \frac{\cos\t}{\sqrt{q^2-1}},
$$
we easily see that the lines $\l_\pm=0$ are as sketched in Fig.
\ref{f4}.

\begin{figure}[h]
\centering
\includegraphics[width=4.5in]{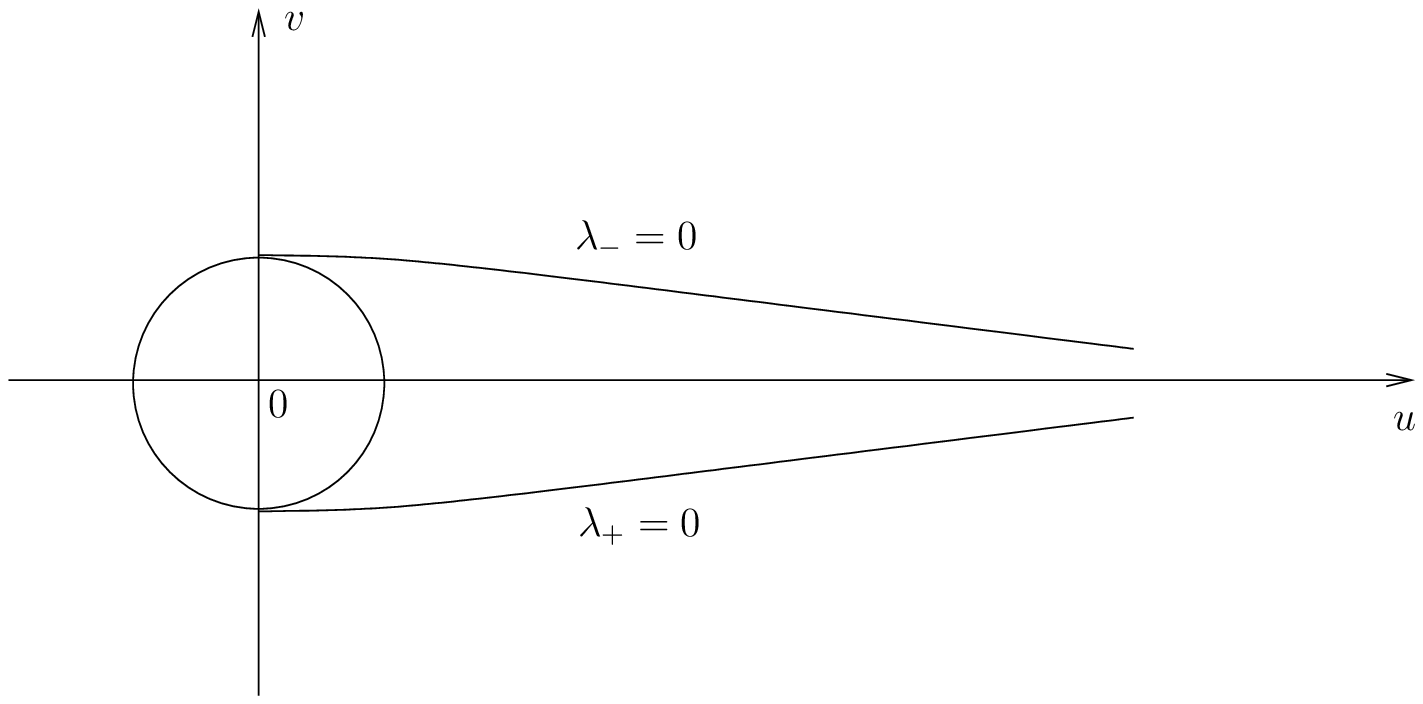}  
\caption{Graphs of $\l_\pm=0$}
 \label{f4}
\end{figure}

Notice that $\l_+>0$ and $\l_-<0$ are in the region above $\l_+=0$
and below $\l_-=0$. If we now super-impose Fig. \ref{f3} on top of
Fig. \ref{f4} and choose $\a$ sufficiently close to $\b$, we see
that there is a region where the four-sided region of Fig. \ref{f3}
is entirely confined in the region above $\l_+=0$ and below $\l_-=0$
in Fig. \ref{f4}. Hence, the parabolic maximum/minimum principles
apply and the four-sided region is an invariant region. For example,
in the half-plane:
\begin{equation} \label{domain-Omega}
\Omega:=\{(x,y)\, :\, x\ge 0, \quad y\in\R\}
\end{equation}
with periodic initial data $(q(0,y),\; \t(0, y))$ prescribed in the
four-sided region,
the  maximum/minimum principles
yield the invariant region for the periodic solution.

%

\bigskip
There is an alternative symmetry about $\t=\frac{\pi}2$.
If we set $\t=\psi+\frac{\pi}2$, then the right-hand side of
\eqref{s1} becomes
$$
\frac1{2\g^2}\left(-\frac{\kappa'}{\r^2
q}+\g^2q\frac{E'}{E}\right)\cos\psi
\pm\frac{1}{2\g^2\r}\left(\frac1{q}\frac{\d\k}{\d y}
-\g^2q\frac{E'}{E}\right)\sin\psi=0.
$$
If $\psi^\pm$ satisfy the above equations, then $\psi^+$ and
$\psi^-$ are symmetric about $\psi=0$, i.e., $\t=\frac{\pi}2$. Look
for the crossing on $\psi=0$ so that
$$
\frac{1}{2\r^2}\left(-\frac{\k'}{\rho^2q}+\g^2q\frac{E'}{E}\right)=0.
$$
With the crossing at $q=\b$, this gives
\begin{equation}\label{ode-2}
\frac{\beta^2-1}{\beta^2}\frac{\k'}{\k}+\frac{E'}{E}=0,
\end{equation}
that is,
$$
|\k(x)|^{\frac{\b^2-1}{\b^2}}E(x)=const.
$$
To exploit the symmetry, we now take $x$ as a space-like variable
and $y$ as a time-like variable and replace $\partial_y^2$ by
$\partial_x^2$. Now it is the interior between the lines $\mu_-=0$
and $\mu_+=0$ that gives the preferred signs $\mu_+>0$ and
$\mu_-<0$, which keep \eqref{g45} and \eqref{g46} parabolic. Hence,
similar to (i), in the case that the initial data  and the metric
$E(x)=G(x), F(x)=0,$ is periodic in $x$, then the periodic solution
will stay in the four-sided invariant region in which the initial
data lies.

\bigskip
All these arguments yield the uniform $L^\infty$ bounds for $(u^\e,
v^\e, p^\e, \rho^\e)$, which implies
$$
|(L^\e, M^\e, N^\e)|\le C,
$$
for some constant $C>0$ depending only on the data and
$\|\gamma\|_{L^\infty}$.

\bigskip
Finally, let us examine the following examples:

\begin{example}\label{5.1a}
Catenoid: $E(x)=(cosh(cx))^{\frac{2}{\beta^2-1}}, \kappa(x)=-\k_0
E(x)^{-\beta^2}$, where $c\ne 0$ and $\k_0>0$ are two constants.
Substitution them into \eqref{ode-1} (where $x$ is time-like and $y$
is space-like) yields
$$
\beta>1.
$$
Of course, \eqref{ode-1} is satisfied when
$$
E(x)=(cosh (cx))^{2(\beta^2-1)}, \quad \kappa(x)=-\k_0
E(x)^{-\frac{\beta^2}{\beta^2-1}},
$$
with $\beta>1$ (where $x$ is space-like and $y$ is time-like).
\end{example}

\begin{example}\label{5.2}
Helicoid: The metric associated with the helicoid is
$$
ds^2=E(dX)^2+(dY)^2
$$
with $E(Y)=\lambda^2+Y^2$ and the Gauss curvature
$$
\kappa=-\frac{\lambda^2}{(\lambda^2+Y^2)^2},
$$
where $\lambda>0$ is a constant. To apply our previous result, for
the special case of isothermal coordinates given in \eqref{5.ex-1},
we first make a change of variables to rewrite the helicoid metric
in isothermal coordinates, i.e., allow $X, Y$ to depend on $x,y$ so
that
$$
ds^2=(EX_x^2+Y_x^2)dx^2+2(EX_xX_y+Y_xY_y)dxdy+(EX_y^2+Y_y^2)dy^2.
$$
Hence, if we set
$$
Y_x=-\sqrt{E} X_y, \quad Y_y=\sqrt{E} X_x,
$$
then
$$
ds^2=E(X_x^2+X_y^2)(dx^2+dy^2),
$$
which gives the metric in isothermal coordinates. The above
equations for $X$ and $Y$ may be rewritten as
$$
-\frac{Y_x}{\sqrt{E}}=X_y, \quad \frac{Y_y}{\sqrt{E}}=X_x,
$$
and with
$$
\phi(Y)=\int\frac{dY}{\sqrt{\lambda^2+Y^2}}
=\ln (Y+\sqrt{\lambda^2+Y^2}),
$$
we have
$$
-\phi_x=X_y, \quad \phi_y=X_x,
$$
i.e., the Cauchy-Riemann equations. A convenient solution is given
by
$$\phi=-x, \quad X=y,$$
which yields
$$
-x=\ln (Y+\sqrt{\lambda^2+Y^2}),
$$
that is,
$$
Y=-\frac12(\lambda^2e^x-e^{-x}).
$$
Thus, in the new $(x,y)$-coordinates,
$$
E=\lambda^2+Y^2=\frac12\lambda^2+\frac14(\lambda^4e^{2x}+e^{-2x}), \quad
\kappa=\frac{-\lambda^2}{(\lambda^2+Y^2)^2} =\frac{-\lambda^2}
{\left(\frac12\lambda^2+\frac14(\lambda^4e^{2x}+e^{-2x})\right)^2},
$$
and
$$
ds^2=\big(\frac12\lambda^2+\frac14(\lambda^4e^{2x}+e^{-2x})\big)
       (dx^2+dy^2).
$$
Hence, we have
$$
-\frac{2E'(x)}{E(x)}=\frac{\kappa'(x)}{\kappa(x)}
$$
and so relation \eqref{ode-1} is satisfied with $\beta=\sqrt{2}$.
\end{example}

\begin{example}\label{5.3}
Torus: The metric for the torus is usually written as
$$
ds^2=E dX^2+b^2 dY^2,
$$
with
$$
E=(a+b\cos Y)^2, \quad \kappa(Y)=\frac{\cos Y}{b(a+b\cos Y)},
$$
where $a>b>0$ are constants. The same argument as given in Example
{\rm 5.2} above yields the metric in isothermal coordinates as
$$ds^2=E(dx^2+dy^2)$$
with
$$
E=(a+b\cos Y)^2=(a+b\cos(\phi^{-1}(x))^2),
$$
where
$$
\phi(Y)=\frac{b}{\sqrt{a^2-b^2}}\arctan
\left(\frac{\sqrt{a^2-b^2}\sin Y}{b+a\cos Y}\right),
$$
$$
\kappa(x)=\frac{\cos Y}{b(a+b\cos Y)} =\frac{\cos
(\phi^{-1}(x))}{b\big(a+b\cos (\phi^{-1}(x))\big)}.
$$
A direct
computation yields
$$
\frac{\kappa'(x)}{\kappa(x)}
=-\frac{a(\phi^{-1}(x))' \tan(\phi^{-1}(x))}{a+b\cos(\phi^{-1}(x))},
\quad \frac{E'(x)}{E(x)} =-\frac{b(\phi^{-1}(x))'
\sin(\phi^{-1}(x))}{a+b\cos(\phi^{-1}(x))},
$$
and the ratio
$$
{\frac{\kappa'(x)}{\kappa(x)}}\left/{\frac{E'(x)}{E(x)}}\right.
=\frac{a}{b\cos(\phi^{-1}(x))}
$$
is not a constant. So \eqref{ode-1} does not hold. Hence our
Proposition {\rm 5.1} will not directly apply to that piece of the
torus possessing negative Gauss curvature.
\end{example}

\medskip
\subsection{$H^{-1}_{loc}$--compactness} \label{Hloc}
We now show how the $H^{-1}_{loc}$--compactness can be achieved for
the viscous periodic approximate solutions via parabolic
regularization.

In \S 5.4, for any initial data in the four-sided region which is
periodic in $y$ with period $P$, we have a uniform $L^\infty$
estimate on $(u^\e,v^\e, p^\e, \rho^\e)$ as the periodic solution to
the viscous equations \eqref{g13}.
From the equations in \eqref{g47b},
we have
\begin{equation}\label{5.32a}
\begin{aligned}
\d_x(\r u)+ \d_y(\r v)&=\frac{v}{q^2}R_1+\frac{u}{q^2}R_2
    +\e\frac1{q}\frac{\d^2}{\d y^2}(\r q)-\e\r(\partial_y\t)^2\\
%
&=B(x,y)+\e\frac1{q}\d^2_y(\r q)-\e\r (\d_y\t)^2,
\end{aligned}
\end{equation}
where
\begin{equation*}
\begin{split}
B(x,y) &=\frac{\r}{q}\sin\t\Big(-\big(\r
q^2\sin^2\t-\frac1{\r}\big)\tilde{\G}_{22}^{(2)}
       -\r q^2\sin(2\t)\tilde{\G}_{12}^{(2)}
       -\big(\r q^2\cos^2\t-\frac1{\r}\big)\tilde{\G}_{11}^{(2)}\Big)\\
&\quad +\frac{\r}{q}\cos\t\Big(-\big(\r
q^2\sin^2\t-\frac1{\r}\big)\tilde{\G}_{22}^{(1)}
       -\r q^2\sin(2\t)\tilde{\G}_{12}^{(1)}
       -\big(\r q^2\cos^2\t-\frac1{\r}\big)\tilde{\G}_{11}^{(1)}\Big).
\end{split}
\end{equation*}
Our $L^\infty$ estimate in \S 5.4 guarantees  that $B(x,y)$ is
uniformly bounded  with respect to $\e$.

Using the periodicity, we have
\begin{eqnarray*}
&&\int_{0}^{x_1}\int_{0}^{P}\frac1{q}\d^2_y(\r q) dydx
=\int_{0}^{x_1}\int_{0}^{P}\frac{\d_y q}{q^2}\d_y(\r q) dydx\\
&&=\int_{0}^{x_1}\int_{0}^{P}\frac{\d_yq}{q^2}\left(-\r^3\d_yq
\right) dydx =-\int_{0}^{x_1}\int_{0}^{P}\frac{\r^3 (\d_y
q)^2}{q^2}dydx.
\end{eqnarray*}

Now integrating both sides of \eqref{5.32a} over $\{(x,y): \, 0\le
x\le x_1, 0\le y\le P\}$,
we find
\begin{eqnarray*}
&&\e\int_{0}^{x_1}\int_{0}^P \Big(\frac{\r^3 (\d_yq)^2}{q^2}+\r
(\d_y\t)^2\Big)dydx\\
&&= \int_{0}^{x_1}\int_0^P B(x,y)\, dydx -\int_{0}^{P}\big((\rho
u)(x_1,y)-(\rho u)(x_0,y)\big)dy\\
&&\le  C,
\end{eqnarray*}
where $C>0$ is independent of $\e$, but may depend on $x_1$ and $P$.
This implies that
$$
\sqrt{\e}\d_y\t, \, \sqrt{\e}\d_y q \qquad\text{are in
$L^2_{loc}(\O)$ uniformly in $\e$}.
$$

Therefore, we have
\begin{proposition} \label{H-1}
{\rm (i)} Consider the viscous system \eqref{g13} in
$\Omega=\{(x,y): \, x\ge 0, y\in\R\}$ with periodic initial data
$(q,\theta)|_{x=0}=(q_0(y),\theta_0(y))$,
then
$$
\sqrt{\e}\d_y q, \; \sqrt{\e}\t_y \qquad\text{are in $L^2_{loc}(\O)$
uniformly in $\e$};
$$
{\rm (ii)} If we replace $\d_y^2$ in system \eqref{g13} by $\d^2_x$,
the initial data $(q,\theta)|_{x=0}=(q_0(y),\theta_0(y))$ by
\begin{equation}\label{initial-x}
(q,\theta)|_{y=0}=(q_0(x),\theta_0(x))
\end{equation}
which is periodic in $x$ with period $P$ and, in addition, we assume
that the metric $E(x)=G(x)$ is also periodic with period $P$,
then the periodic solution with period $P$ satisfies that
$$
\sqrt{\e}\d_x q, \; \sqrt{\e}\d_x\t \qquad\text{are in
$L^2_{loc}(\O)$ uniformly in $\e$},
$$
where $\Omega=\{(x,y): \, x_0\le x <x_1, y> 0\}$.
\end{proposition}

Using Proposition \ref{H-1} and the viscous system \eqref{g13}, we
conclude that
$$
\d_x\tilde{M}^\e-\d_y\tilde{L}^\e, \quad \d_x\tilde{N}^\e-\d_y
\tilde{M}^\e \quad\text{are compact in } H^{-1}_{loc}(\O).
$$

Since $\gamma\in C^{1}$, we conclude that $(L^\e, M^\e,
N^\e)=\gamma(\tilde{L}^\e, \tilde{M}^\e, \tilde{N}^\e)$ satisfies
Framework (A) in \S 4. Then the compensated compactness framework
(Theorem 4.1) implies that there is a subsequence (still labeled)
$(L^\varepsilon, M^\varepsilon, N^\varepsilon)(x,y)$ that converges
weak-star to $(\bar{L}, \bar{M}, \bar{N})$ as $\varepsilon\to 0$
such that the limit $(\bar{L}, \bar{M}, \bar{N})$ is a bounded,
periodic weak solution to the Gauss-Codazzi system
\eqref{g1}--\eqref{g2}.
%
%
%
Therefore, $(\bar{L},\bar{M},\bar{N})$ is a weak solution of
\eqref{g1}--\eqref{g2}. We summarize this as Proposition 5.2.

\begin{proposition}\label{H-2}
For either initial value problem {\rm (i)} or {\rm (ii)} of
Proposition {\rm \ref{H-1}}, $(L^\e, M^\e, N^\e)$ possesses a
weak-star convergent subsequence which converges to a periodic weak
solution of the associated initial value problem for the
Gauss-Codazzi system \eqref{g1}--\eqref{g2} when $\e\to 0$.
\end{proposition}

\subsection{Existence of isometric immersions: Main theorem and examples}

We now focus on the case \eqref{5.ex-1}:
$$
F=0, \,\, \text{and}\,\, E=G \text{ depends only on } x.
$$
to state an existence result for isometric immersions and analyze
examples for this case.

Let us look for a special solution:
$$
(\t, q)=(0, \beta)   \text{ (constant state)},
$$
for the Gauss-Codazzi system for the case \eqref{5.ex-1}. In this
case,
$$
 \tilde{\G}_{11}^{(1)}=\frac{E'}{2E}+\frac{\g'}{\g},\quad
\tilde{\G}_{12}^{(1)}=0,\quad
 \tilde{\G}_{22}^{(1)}=-\frac{E'}{2E}; \quad
 \tilde{\G}_{11}^{(2)}=0, \quad
  \tilde{\G}_{12}^{(2)}=\frac{E'}{2E}+\frac{\g'}{2\g},\quad
\tilde{\G}_{22}^{(2)}=0,
$$
and the Gauss-Codazzi system \eqref{g3a} becomes
$$
\partial_x\big(\frac1{\r}\big)=- p\frac{E'}{2E}+(\rho
q^2+p)(\frac{E'}{2E}+\frac{\g'}{\g}) =\rho q^2\frac{E'}{2E} +(\rho
q^2+p)\frac{\g'}{\g}
 =\rho
q^2\frac{E'}{2E} +\rho\frac{\g'}{\g},
%
\quad
  \r=\frac{1}{\sqrt{q^2-1}}.
$$
When $q(x)\equiv \b$, this reduces to
$$
\frac{1}{\b^2}\frac{\g'}{\g}=-\frac{E'}{2E},
$$
or \begin{equation} \label{ourcase}
\frac{1}{\b^2}\frac{\k'(x)}{\k(x)}=-\frac{E'}{E}.
\end{equation}
Hence, $q(x)=\b$ becomes an exact solution precisely in this special
case.
Our theorem given below shows that, in fact, we can satisfy the
prescribed initial conditions in this special case and that, for
this choice of $E, F, G$, there exists a weak solution for arbitrary
bounded data in our diamond-shaped region when $\a\in (1,\beta)$
(see Fig. 3).

%

\medskip
Consider the initial value problem for the Gauss-Codazzi system
\eqref{g1}--\eqref{g2} with initial data
\begin{equation}\label{6.1}
(q,\theta)|_{x=0}=(q_0(y), \theta_0(y)), \qquad y\in\R,
\end{equation}
or
\begin{equation}\label{6.2}
(q,\theta)|_{y=0}=(q_0(x), \theta_0(x)), \qquad x\in\R.
\end{equation}

Our next result shows that, for this choice of $E, F, G$, there
exists a weak solution for arbitrary bounded initial data in our
diamond-shaped region  when $\a\in (1, \beta)$
(see Fig. 3).

\begin{theorem} Assume that the initial data \eqref{6.1}, or \eqref{6.2}, is
$L^\infty$ and lies in the diamond-shaped region of Figs. {\rm
3--4}. Then

{\rm (i)} The Gauss-Codazzi system \eqref{g1}--\eqref{g2} has a weak
solution with the initial data $(q, \t)|_{x=0}=(q_0(y), \t_0(y))$.
This case includes Example {\rm 5.1} for the catenoid with the
metric $E(x)=G(x)=(cosh(cx))^{\frac{2}{\beta^2-1}}$, $F(x)=0$, $c\ne
0$, and $\beta>1$, and Example {\rm 5.2} for the helicoid (in
isothermal coordinates) with
$E(x)=G(x)=\frac{1}{2}\lambda^2+\frac{1}{4}(\lambda^4e^{2x}+e^{-2x})$,
$F(x)=0$, and $\beta=\sqrt{2}$.
%

{\rm (ii)} The Gauss-Codazzi system \eqref{g1}--\eqref{g2} has a
weak solution with the initial data
$(q,\t)|_{y=0}=(q_0(x),\t_0(x))$. This case includes the catenoid
with the metric $E(x)=G(x)=(cosh(cx))^{2(\beta^2-1)}$, $F(x)=0$,
$c\ne 0$, and $\beta>1$, and the helicoid (in isothermal
coordinates) with
$E(x)=G(x)=\frac{1}{2}\lambda^2+\frac{1}{4}(\lambda^4e^{2x}+e^{-2x})$,
$F(x)=0$, and $\beta=\sqrt{2}$.
%
\end{theorem}

\begin{proof} We start with case (i).

\smallskip
{\it Step 1.} For the initial data $(q_0(y), \theta_0(y))$, we can
find $(q_0^P(y), \theta_0^P(y))$ for $P>0$ such that

\smallskip
(i) \, $q^P_0, \theta^P_0\in C^1(\R), \quad q_0^P,\theta^P
\,\,\mbox{are periodic with period $P$}$;


(ii) \, $ q_0^P\to q_0, \theta_0^P\to \theta_0 \quad a.e.\,\,
\mbox{in $\R$ and weakly in $L^\infty(\R)$ as $P\to \infty$}$.

\smallskip
\noindent In particular, the functions $q_0^P$ and $\theta_0^P$ are
bounded in $L^\infty(\R)$, and $(q_0^P, \theta_0^P)$ converges to
$(q_0, \theta_0)$ in $L^p_{loc}(\R)$, $p\in [1, \infty)$, as $P\to
\infty$.

This can be achieved by the standard symmetric mollification
procedure: First truncate the initial data $(W_{-,0}, W_{+,0})=
(W_-(q_0,\theta_0), W_+(q_0,\theta_0))$ in the interval
$-\frac{P}{2}\le y\le \frac{P}{2}$ and make the periodic extension
to the whole space $y\in\R$, and then take the standard symmetric
mollification approximation to get the $C^\infty$ approximate
sequence $(W_{-,0}^P, W_{+,0}^P)$ that yields the corresponding
$C^\infty$ approximate sequence
$$
(q_0^P, \theta_0^P)=((\cos(\frac{W_{+,0}^P-W_{-,0}^P}{2}))^{-1},
\frac{W_{+,0}^P+W_{-,0}^P}{2})
$$
converging to $(q_0,\theta_0)=(
(\cos(\frac{W_{+,0}-W_{-,0}}{2}))^{-1}, \frac{W_{+,0}+W_{-,0}}{2})$
a.e. in $\R$ as $P\to \infty$. Since the standard symmetric
mollification is an average-smoothing operator,
then the approximate sequence
$(q_0^P, \theta^P_0)$ still lies in our diamond-shaped region of
Figs. 3--4.

\medskip
{\it Step 2}. Following the arguments in Section 5.4, we can
established the uniform $L^\infty$ apriori estimates for the
corresponding viscous solutions in two parameters $\epsilon>0$ and
$P>0$. For fixed $P$, then we can show that there exists a unique
periodic viscous solution with period $P$ to the parabolic system
\eqref{g13}, which can be achieved by combining the standard local
existence theorem with the $L^\infty$ estimates. The
$H^{-1}_{loc}$--compactness follows from the argument in \S
\ref{Hloc}.

For fixed $P$, letting the viscous coefficient $\epsilon$ tend $0$,
we employ Proposition 5.2 to obtain the global periodic weak
solution $(L^P, M^P, N^P)$ of the Gauss-Codazzi system
\eqref{g1}--\eqref{g2}, periodic in $y$ with period $P$, in the half
plane $\{(x,y)\,:\, x\ge 0, y\in \R\}$.

\medskip
{\it Step 3}. Since the sequence of periodic solutions $(L^P, M^P,
N^P)$ still stays in the invariant region, which yields the uniform
$L^\infty$ bound in $P$ as $P\to\infty$. This uniform bound also
yields the $H^{-1}_{loc}$--compactness of
$$
\d_x\tilde{M}^P-\d_y\tilde{L}^P, \quad \d_x\tilde{N}^P-\d_y
\tilde{M}^P.
$$

Using the compensated compactness framework (Theorem 4.1) again and
letting $P\to \infty$, we obtain a global weak solution $(L, M, N)$
 of the Gauss-Codazzi system \eqref{g1}--\eqref{g2} in the half
plane $\{(x,y)\,:\, x\ge 0, y\in \R\}$.

\medskip
{\it Step 4}. For (ii), it also requires the periodic approximation
for the metric $E(x)=G(x)$ by $E^P(x)=G^P(x)$ with period $P$ in
$x$, besides the period approximation for the initial data. This can
be achieved as follows: First truncate the function
$a(x):=\frac{E'(x)}{E(x)}$ in the interval $-\frac{P}{4}\le x\le
\frac{P}{4}$ and make odd extension along the lines $x=-\frac{P}{4},
\frac{P}{4}$ respectively. Then make the periodic extension from the
interval $-\frac{P}{2}\le x\le \frac{P}{2}$ to the whole space
$x\in\R$ with period $P$ and take the standard symmetric
mollification approximation to get the $C^\infty$ approximate
sequence $a^P(x)$ with zero mean over each period  (i.e.
$\int_{-\frac{P}{2}}^{\frac{P}{2}}a^P(x)dx=0$) that yields the
corresponding $C^\infty$ approximate sequence $E^P(x)$ with period
$P$ uniquely determined by the differential equation:
$$
Y'(x)=a^P(x) Y(x), \qquad Y|_{x=0}=E(0),
$$
and then take the $C^\infty$ approximate sequence $\kappa^P(x)$ with
period $P$  determined by the differential equation:
$$
K'(x)=-\frac{\beta^2}{\beta^2-1}a^P(x) K(x), \qquad
K|_{x=0}=\kappa(0)
$$
for the corresponding $\beta>1$ if needed.

Substitute $E^P=G^P$ and $\kappa^P$ into the viscous Gauss-Codazzi
system so that the equations in the system is periodic in $x$ with
period $P$. For fixed $P$, the same argument as that for the first
case yields the global period weak solution $(L^P, M^P, N^P)$ of the
Gauss-Codazzi system \eqref{g1}--\eqref{g2} (with $E(x)$ replaced by
$E^P(x)$, periodic in $x$ with period $P$, in the half plane
$\{(x,y)\,:\, x\in \R, y\ge 0\}$. Then, using the same argument,
noting the strong convergence of $E^P$ to $E$, and letting $P\to
\infty$, we again obtain a global weak solution $(L, M, N)$
 of the Gauss-Codazzi system \eqref{g1}--\eqref{g2} in the half
plane $\{(x,y)\,:\, x\in \R, y\ge 0\}$.
\end{proof}

If we repeat the similar argument for Theorem 5.1 through the
corresponding vanishing viscosity method on the domain
$\Omega=\{(x,y)\,:\, x< 0, y\in \R\}$ for problem \eqref{6.1} and
the domain $\Omega=\{(x,y)\,:\, x\in\R, y< 0\}$ for problem
\eqref{6.2}, we obtain again a weak solution. Together, they form a
weak solution in $\R^2$. As before, the associated immersion is in
$C^{1,1}$.

\begin{theorem}
Assume that the initial data \eqref{6.1}, or \eqref{6.2}, is
$L^\infty$ and lies in the diamond-shaped region of Figs. {\rm 3--4}
for the case of the catenoid or helicoid metric (as given in Theorem
{\rm 5.1}), then the initial value problem \eqref{g1}--\eqref{g2}
and \eqref{6.1}, or \eqref{6.2}, has a weak solution in
$L^\infty(\R^2)$. This yields a $C^{1,1}(\R^2)$ immersion of the
Riemannian manifold into $\R^3$.
\end{theorem}

\begin{remark}
The catenoid with circular cross-section is sketched in Fig. 5. Our
theorem asserts the existence of a $C^{1,1}$-surface for the
associated metric for a class of non-circular cross-sections
prescribed at $x=0$. Similarly, the $C^{1,1}$-helicoid in the
$(x,y,z)$--coordinates is sketched in Fig. 6.
\end{remark}

\begin{figure}[h]
\centering
\includegraphics[width=2in]{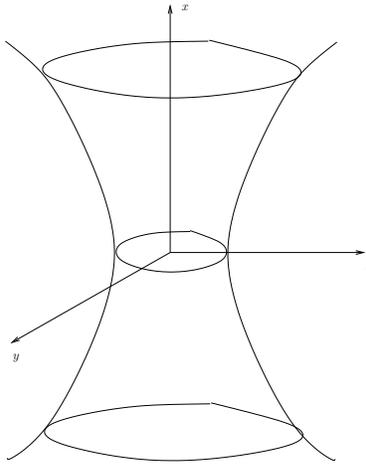}  
\caption{$C^{1,1}$-catenoid in the $(x,y,z)$--coordinates}
 \label{f5}
\end{figure}

\begin{figure}[h]
\centering
\includegraphics[width=2in]{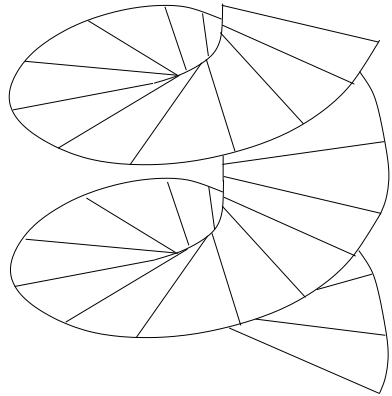}  
\caption{$C^{1,1}$-helicoid in the $(x,y,z)$--coordinates}
 \label{f6}
\end{figure}

\bigskip
{\bf Acknowledgments.} This paper was completed when the authors
attended the ``Workshop on Nonlinear PDEs of Mixed Type Arising in
Mechanics and Geometry'', which was held at the American Institute
of Mathematics, Palo Alto, California, March 17--21, 2008. Gui-Qiang
Chen's research was supported in part by the National Science
Foundation under Grants DMS-0807551, DMS-0720925, DMS-0505473, and
an Alexander von Humboldt Foundation Fellowship. Marshall Slemrod's
research was supported in part by the National Science Foundation
under Grant DMS-0243722. Dehua Wang's research was supported in part
by the National Science Foundation under Grants DMS-0244487,
DMS-0604362, and the Office of Naval Research Grant
N00014-01-1-0446.

\end{document}